\documentclass[10pt,a4paper]{amsart}
\usepackage{amsmath}
\usepackage{amssymb}
\usepackage{amsthm}
\usepackage{amsfonts}
\usepackage{tikz}
\usepackage[all]{xy}
\usetikzlibrary{matrix}
\usepackage{indentfirst}
\usepackage[latin1]{inputenc}
\usepackage{enumitem}
\usepackage{cancel}

\DeclareSymbolFont{largesymbols}{OMX}{yhex}{m}{n}
\DeclareMathAccent{\widehat}{\mathord}{largesymbols}{"62}

\usepackage[a4paper, margin=1in]{geometry}

\usepackage[bookmarks=false,colorlinks=true,linkcolor=black,citecolor=black,filecolor=black,urlcolor=black]{hyperref}

\newcommand{\Q}{{\mathbb Q}}
\newcommand{\Z}{{\mathbb Z}}
\newcommand{\C}{{\mathbb C}}
\newcommand{\F}{{\mathbb F}}
\newcommand{\R}{{\mathbb R}}

\newcommand{\U}{\mathcal{U}}
\newcommand{\Gal}{\textnormal{Gal}}
\newcommand{\Irr}{\textnormal{Irr}}
\newcommand{\Cen}{\textnormal{Cen}}

\newcommand{\Aut}{\textnormal{Aut}}

\newcommand{\tr}{{\textnormal{tr}}}
\newcommand{\End}{\textnormal{End}}
\newcommand{\suma}[1]{\widehat{#1}}
\newcommand{\inv}{{^{-1}}}
\newcommand{\GEN}[1]{\left\langle #1 \right\rangle}
\newcommand{\Stab}{\textnormal{Stab}}
\newcommand{\Orb}{\textnormal{Orb}}

\theoremstyle{plain}
\newtheorem{theorem}{Theorem}[section]
\newtheorem{definition}[theorem]{Definition}
\newtheorem{lemma}[theorem]{Lemma}
\newtheorem{corollary}[theorem]{Corollary}
\newtheorem{proposition}[theorem]{Proposition}
\theoremstyle{definition}
\newtheorem{remark}[theorem]{Remark}

\newtheorem{example}[theorem]{Example}

\begin{document}

\title[On idempotents and the number of simple components of semisimple group algebras]{On idempotents and the number of simple components of semisimple group algebras}

\author{Gabriela Olteanu} 
\address{Department of Statistics-Forecasts-Mathematics, Babe\c s-Bolyai University,
Str. T. Mihali 58-60, 400591 Cluj-Napoca, Romania}
\email{gabriela.olteanu@econ.ubbcluj.ro}

\author{Inneke Van Gelder}
\address{Department of Mathematics, Vrije Universiteit Brussel,
Pleinlaan 2, 1050 Brussels, Belgium}
\email{ivgelder@vub.ac.be}

\thanks{The research is supported by the grant PN-II-ID-PCE-2012-4-0100 and by the Research Foundation Flanders (FWO - Vlaanderen).}

\date{\today}

\maketitle

\begin{abstract}
We describe the primitive central idempotents of the group algebra over a number field of finite monomial groups. We give also a description of the Wedderburn decomposition of the group algebra over a number field for finite strongly monomial groups.  Further, for this class of group algebras, we describe when the number of simple components agrees with the number of simple components of the rational group algebra. Finally, we give a formula for the rank of the central units of the group ring over the ring of integers of a number field for a strongly monomial group.
\end{abstract}

\section{Introduction}
In 2004, Olivieri, del R\'io and Sim\'on \cite{Olivieri2004} showed a method to describe the primitive central idempotents of $\Q G$ for finite monomial groups. Furthermore, they were able to provide information on the Wedderburn decomposition of $\Q G$ for strongly monomial groups (including abelian-by-supersolvable groups). In 2007, Broche and del R\'io \cite{Broche2007} did the same for finite semisimple group algebras. 
Let $F$ be a number field. The main aim of this work is to show that similar methods can be used to describe the primitive central idempotents of $FG$ for finite monomial groups and to describe the Wedderburn decomposition of $FG$ for finite strongly monomial groups. This description allows us to compute the number of simple components of a group algebra $FG$ for finite strongly monomial groups and to see when this number is minimal, i.e. equals the number of simple components of $\Q G$. We find necessary and sufficient conditions on the presentation of the group $G$ and the number field $F$ for the number of simple components of $FG$ to be minimal, for finite abelian groups and some metacyclic groups.
An analogue is given for finite semisimple group algebras that agrees with the formula given by \cite{Ferraz2007} for finite abelian groups. 

Let $R$ be the ring of integers of a number field $F$. For a finite group $G$ we denote by $\U(R G)$ the unit group of the group ring $R G$. Its group of central units is denoted by $\mathcal{Z}(\U(R G))$. It follows from Dirichlet's Unit Theorem that $\mathcal{Z}(\U(R G))=T \times A$, where $T$ is a finite group and $A$ is a finitely generated free abelian group.
In the last section, we compute the rank of $A$, which is called the rank of $\mathcal{Z}(\U(R G))$ for $G$ a strongly monomial group. 

We first fix some notations. For two coprime integers $r$ and $m$, we denote the multiplicative order of $r$ modulo $m$ by $o_m(r)$. A group $G$ will always be assumed to be finite. Let $F$ be a number field. 

If $\alpha \in FG$ and $g \in G$ then $\alpha^g = g\inv \alpha g$ and $\Cen_G(\alpha)$ denotes the centralizer of $\alpha$ in $G$. The notation $H\leq G$ (respectively, $H \unlhd G$) means
that $H$ is a subgroup (resp., normal subgroup) of $G$. If $H\leq G$ then $N_G(H)$ denotes the normalizer of $H$ in $G$ and we set $\suma{H} = \frac{1}{|H|}\sum_{h\in H}h$. If
$g\in G$ then $\suma{g}=\suma{\GEN{g}}$.

By assumption, all the characters of $G$ are considered as complex characters. For an irreducible character $\chi$ of $G$, let $F(\chi)$ denote the field of character values of $\chi$, $e(\chi)=\frac{1}{|G|}\sum_{g\in G}\chi(g\inv)g$ the primitive central idempotent of $\C G$ associated to $\chi$ and $e_F(\chi)$ the unique primitive central idempotent of $FG$ such that $\chi(e_F(\chi))\neq 0$.

The Galois group $\Gal(F(\chi)/F)$ acts on $F(\chi) G$ by acting on the coefficients, that is 
$$\sigma \cdot \sum_{g\in G} a_gg = \sum_{g\in G} \sigma(a_g)g,$$ for $\sigma\in\Gal(F(\chi)/F)$ and $a_g\in F(\chi)$. 
Following \cite{Yamada1973}, we have \begin{eqnarray}\label{yamada}  e_{F}(\chi) = \sum_{\sigma \in \Gal(F(\chi)/F)} \sigma \cdot e(\chi).\end{eqnarray}  

If $G=\GEN{g}$ is cyclic of order $k$, then the irreducible characters are all linear and are defined by the image of a generator of $G$. Therefore the set $G^*=\Irr(G)$ of irreducible characters of $G$ is a group and the map $\phi:\Z/k\Z\rightarrow G^*$ given by $\phi(m)(g)=\zeta_k^m$ is a group homomorphism. The generators of $G^*$ are precisely the faithful representations of $G$. Let $\mathcal{C}_F(G)$ denote the set of orbits of the faithful characters of $G$ under the action of $\Gal(F(\zeta_k)/F)$. Note that for any faithful character $\chi$ of $G$, $F(\chi)=F(\zeta_k)$.

Each automorphism $\sigma\in\Gal(F(\zeta_k)/F)$ is completely determined by its action on $\zeta_k$, and is given by $\sigma(\zeta_k)=\zeta_k^t$, where $t$ is an integer uniquely determined modulo $k$. In this way, one gets the following morphisms 
$$\SelectTips{cm}{}
\xymatrix{
\Gal(F(\zeta_k)/F) \hspace{2mm} \ar@{^{(}->}[r] \ar[d]_{\simeq} & \hspace{1mm}\Gal(\Q(\zeta_k)/\Q) \ar[d]^{\simeq} \\ 
I_k(F) \hspace{2mm} \ar@{^{(}->}[r] & \hspace{1mm} \U(\Z/k\Z) 
}$$
where we denote the image of $\Gal(F(\zeta_k)/F)$ in $ \U(\Z/k\Z)$ by $I_k(F)$ (see \cite[\S21A]{1981CurtisReiner} for notations).
In this setting, the sets in $\mathcal{C}_F(G)$ corresponds one-to-one to the orbits under the action of $I_k(F)$ on $\U(\Z/k\Z)$ by multiplication.

For $x$ and $y$ in $G$, we say that $x$ and $y$ are $F$-conjugate if $x$ is conjugate in $G$ to $y^t$ for some $t\in I_e(F)$, where $e$ is the exponent of $G$. Hence, when $G$ is cyclic $\mathcal{C}_F(G)$ also corresponds to the $F$-conjugacy classes of $G$ containing generators of $G$.

\section{Primitive central idempotents associated to monomial irreducible characters}

Let $N\unlhd G$ be such that $G/N$ is cyclic of order $k$ and $C\in\mathcal{C}_F(G/N)$. If $\chi\in C$ and $\tr=\tr_{F(\zeta_k)/F}$ denotes the field trace of the Galois extension $F(\zeta_k)/F$, then we set 
$$\varepsilon_C(G,N)=\frac{1}{|G|}\sum_{g\in G} \tr(\chi(gN))g\inv=\frac{1}{|G|}\sum_{g\in G}\sum_{\psi\in C}\psi(gN)g\inv=[G:N]\inv \suma{N}\sum_{X\in G/N}\sum_{\psi\in C} \psi(X)g_X^{-1},$$ where $g_X$ denotes a representative of $X\in G/N$. Note that $\varepsilon_C(G,N)$ does not depend on the choice of $\chi\in C$. 

Let $K\unlhd H\leq G$ such that $H/K$ is cyclic and $C\in\mathcal{C}_F(H/K)$. Then $e_C(G,H,K)$ denotes the sum of the different $G$-conjugates of $\varepsilon_C(H,K)$.

\begin{proposition}\label{finiteabelian}
If $G$ is a finite abelian group of order $n$ and $F$ is a number field, then the map $(N,C)\mapsto \varepsilon_C(G,N)$ is a bijection from the set of pairs $(N,C)$ with $N\unlhd G$, such that $G/N$ is cyclic, and $C\in\mathcal{C}_F(G/N)$ to the set of primitive central idempotents of $F G$. Furthermore, for every $N\unlhd G$ and $C\in\mathcal{C}_F(G/N)$, $F G\varepsilon_C(G,N)\simeq F(\zeta_k)$, where $k=[G:N]$.
\end{proposition}
\begin{proof}
 If $e$ is a primitive central idempotent of $FG$, then there exists an irreducible character $\psi$ of $G$ such that $e=e_F(\psi)$. Since $G$ is abelian, $\psi$ is linear. Let $N$ denote the kernel of $\psi$ and let $\chi$ be the faithful character of $G/N$ given by $\chi(gN)=\psi(g)$. Then $G/N$ is cyclic, say of order $k$, $F(\psi)=F(\zeta_k)$ and the orbit of $\chi$ under the action of $\Gal(F(\zeta_k)/F)$ is an element of $\mathcal{C}_F(G/N)$. Furthermore 
\begin{eqnarray}\label{eq1}
\begin{gathered}
 e_F(\psi)= \sum_{\sigma\in\Gal(F(\psi)/F)} \sigma\cdot e(\psi)=\frac{1}{|G|}\sum_{\sigma\in \Gal(F(\psi)/F)}\sum_{g\in G} \sigma(\psi(g))g\inv\\
= \frac{1}{|G|}\sum_{\sigma\in \Gal(F(\psi)/F)}\sum_{g\in G} \sigma(\chi(gN))g\inv=\frac{1}{|G|}\sum_{g\in G}\tr_{F(\zeta_k)/F}(\chi(gN))g\inv=\varepsilon_C(G,N).
\end{gathered}
\end{eqnarray}
This shows that the map is surjective and $FG\varepsilon_C(G,N)=FGe_F(\psi)\simeq F(\zeta_k)$. 

Assume now that $\varepsilon_{C_1}(G,N_1)=\varepsilon_{C_2}(G,N_2)$ with $N_i\unlhd G$ and $C_i\in \mathcal{C}_F(G/N_i)$. Take $\chi_i\in C_i$. Let $\pi_i:G\rightarrow G/N_i$ be the canonical projections and $\psi_i=\chi_i\circ \pi_i$. Then $\psi_i$ are irreducible characters of $G$. By~(\ref{eq1}), $e_F(\psi_1)=\varepsilon_{C_1}(G,N_1)=\varepsilon_{C_2}(G,N_2)=e_F(\psi_2)$ and $F(\psi_1)=F(\psi_2)$. If $K=F(\psi_i)$, then there exists a $\sigma\in \Gal(K/F)$ such that $\psi_2=\sigma \circ \psi_1$ and hence $N_1=\ker \psi_1=\ker \psi_2=N_2$. 
Now let $\pi\inv$ be a right inverse of $\pi_1=\pi_2$. Then $\chi_2=\psi_2\circ\pi\inv=\sigma\circ\psi_1\circ\pi\inv=\sigma\circ\chi_1$ and hence $C_1=C_2$. This shows that the map is injective.
 \end{proof}

\begin{corollary}\label{cyclicquotient}
 If $G$ is a finite group with normal subgroup $N$ such that $G/N$ is cyclic of order $k$ and $F$ is a number field, then $\varepsilon_C(G,N)$ is a primitive central idempotent of $FG$ and 
$F G\varepsilon_C(G,N)\simeq F(\zeta_k)$. Furthermore, if $D\in \mathcal{C}_F(G/N)$, then $\varepsilon_C(G,N)=\varepsilon_D(G,N)$ if and only if $C=D$.
\end{corollary}
\begin{proof}
 The natural epimorphism $G\rightarrow G/N$ induces a ring isomorphism $\phi: FG\suma{N}\simeq F(G/N)$. Since $\varepsilon_C(G,N)\in FG\suma{N}$ and $\phi(\varepsilon_C(G,N))=\varepsilon_C(G/N,1)$ is a primitive central idempotent of $F(G/N)$ by Proposition~\ref{finiteabelian}, also $\varepsilon_C(G,N)$ is a primitive central idempotent in $FG$. Moreover, again by Proposition~\ref{finiteabelian}, $FG\varepsilon_C(G,N)\simeq F(G/N)\varepsilon_C(G/N,1)\simeq F(\zeta_k)$ and $\varepsilon_C(G,N)=\varepsilon_D(G,N)$ if and only if $\varepsilon_C(G/N,1)=\varepsilon_D(G/N,1)$ if and only if $C=D$.
 \end{proof}

Let $H$ be a subgroup of $G$, $\psi$ a linear character of $H$ and $g\in G$. Then $\psi^g$ denotes the character of $H^g$ given by $\psi^g(h^g)=\psi(h)$. Since $\ker(\psi^g)=\ker(\psi)^g$, the map $\psi\mapsto \psi^g$ is a bijection between linear characters of $H$ with kernel $K$ and linear characters of $H^g$ with kernel $K^g$. This map induces a bijection $\mathcal{C}_F(H/K)\rightarrow\mathcal{C}_F(H^g/K^g): C\mapsto C^g$. In this sense, the following equation is easy to verify:
\begin{eqnarray}\label{eq2}
\varepsilon_C(H,K)^g=\varepsilon_{C^g}(H^g, K^g).
\end{eqnarray}

Now let $K\unlhd H\leq G$ be such that $H/K$ is cyclic of order $k$. Then $N=N_G(H)\cap N_G(K)$ acts on $\mathcal{C}_F(H/K)$ by the rule $$(\{\chi^i\mid i\in I_{k}(F)\},g)\mapsto \{(\chi^g)^i\mid i\in I_{k}(F)\}.$$ 
Similarly, $N$ acts on the $F$-conjugacy classes of $H/K$ containing generators of $H/K$.
Consider the stabilizer of a $C\in\mathcal{C}_F(H/K)$, take $\chi\in C$ and fix $hK$ a generator of $H/K$,
\begin{eqnarray*}
 \Stab_N(C) &=& \{ g\in N\mid \chi^g=\chi^i \mbox{ for some } i\in I_k(F)\}\\
&=& \{g\in N \mid \chi(g\inv hgK)=\chi(h^iK)\mbox{ for some } i\in I_k(F)\}\\
&=& \{g\in N\mid g\inv hgK=h^iK \mbox{ for some } i\in I_k(F)\}\\
&=& \Stab_N(\overline{hK}),
\end{eqnarray*}
where $\overline{hK}$ denotes the $F$-conjugacy class of $hK$.
Note that we used that $\chi$ is faithful in the third equality. Note that $\Stab_N(\overline{hK})$ is independent on the choice of generator of $H/K$. Indeed, let $h_1K$ and $h_2K$ be generators such that $h_1K=h_2^jK$. Then 
\begin{eqnarray*}
\Stab_N(\overline{h_1K}) &=& \{g\in N\mid g\inv h_1gK=h_1^iK \mbox{ for some } i\in I_k(F)\}\\ &=& \{g\in N\mid (g\inv h_2gK)^j=(h_2^{i}K)^j \mbox{ for some } i\in I_k(F)\}\\
&\supseteq& \{g\in N\mid g\inv h_2gK=h_2^{i}K \mbox{ for some } i\in I_k(F)\} = \Stab_N(\overline{h_2K}),
\end{eqnarray*}
which gives an equality while reversing the role of $h_1$ and $h_2$. Hence $\Stab_N(C)$ is independent on the choice of $C\in\mathcal{C}_F(H/K)$ and hence it is the stabilizer of any $C\in\mathcal{C}_F(H/K)$ under the action of $N$. We denote this stabilizer by $E_F(G,H/K)$. Note that $E_{\Q}(G,H/K)=N$.

\begin{definition}
A pair $(H,K)$ of subgroups of $G$ is called a Shoda pair if it satisfies the following conditions:
\begin{enumerate}
\item $K \unlhd H$,
\item $H/K$ is cyclic,
\item if  $g\in G$ and $[H,g]\cap H \subseteq K$, then $g\in H$.
\end{enumerate}
\end{definition}

Let $H$ be a subgroup of $G$ and $\chi$ a character of $H$. Then $\chi^G$ is given by $$\chi^G(g)=\frac{1}{|H|}\sum_{x\in G}\chi^{\circ}(xgx^{-1}),$$ where $\chi^{\circ}$ is defined by $\chi^{\circ}(h)=\chi(h)$ if $h\in H$ and $\chi^{\circ}(y)=0$ if $y\notin H$. It is well known that $\chi^G$ is a character of $G$ and we call it the induced character on $G$ (see \cite[Corollary 5.3]{Isaacs1976}).

A character $\chi$ of $G$ is called monomial if there exist a subgroup $H\leq G$ and a linear character $\psi$ of $H$ such that $\chi = \psi^G$, the induced character on $G$. The group $G$ is called monomial if all its irreducible characters are monomial.

Then we can rephrase an old theorem of Shoda \cite{Shoda1933} as follows.
\begin{proposition}[Shoda] If $\chi$ is a linear character of a subgroup $H$ of $G$ with kernel $K$, then the induced character $\chi^G$ is irreducible if and only if $(H,K)$ is a Shoda pair. 
\end{proposition}

Hence for monomial groups $G$, all irreducible characters are associated with Shoda pairs of $G$. The next theorem provides a description of the primitive central idempotent given by a Shoda pair $(H,K)$ of a group $G$ as a multiple of $e_C(G,H,K)$.

\begin{theorem}
Let $G$ be a finite group and $(H,K)$ a Shoda pair of $G$. Let $F$ be a number field.
Let $\chi$ be a linear character of $H$ with kernel $K$ and $C$ be the orbit of $\chi$ under the action of $\Gal(F(\chi)/F)$. Then $\chi^G$ is irreducible and the primitive central idempotent of $F G$ associated to $\chi^G$ is 
$$e_{F}(\chi^G)=\frac{[\Cen_G(\varepsilon_C(H,K)):H]}{[F(\chi):F(\chi^G)]}e_C(G,H,K).$$
\end{theorem}
\begin{proof}
Let $e=e(\chi)$. Let $\Gal(F(\chi)/F)=\{\sigma_1,\ldots,\sigma_n\}$ and $T=\{g_1,\ldots,g_m\}$ be a right transversal of $H$ in $G$. Then 
\begin{eqnarray*}
e(\chi^G) &=& \frac{1}{|G|} \sum_{g\in G} \chi^G(1)\chi^G(g^{-1}) g \\
&=& \frac{1}{|G|}\sum_{g\in G} \frac{|G|}{|H|}\chi(1) \left( \sum_{i=1}^m \chi^{\circ}(g_ig^{-1}g_i^{-1}) \right)g\\
&=& \frac{1}{|H|}\sum_{i=1}^m\sum_{h\in H}\chi(h^{-1}) g_i^{-1}hg_i\\
&=& \sum_{i=1}^m e\cdot g_i.
\end{eqnarray*}

Consider the table  
$$ \begin{array}{cccc|c}
\sigma_1\cdot e\cdot g_1 & \sigma_1\cdot e\cdot g_2 & \cdots & \sigma_1\cdot e\cdot g_m & \sigma_1\cdot e(\chi^G) \\
\sigma_2\cdot e\cdot g_1 & \sigma_2\cdot e\cdot g_2 & \cdots & \sigma_2\cdot e\cdot g_m & \sigma_2\cdot e(\chi^G)\\
\cdots & \cdots & \cdots & \cdots & \cdots \\
\sigma_n\cdot e\cdot g_1 & \sigma_n\cdot e\cdot g_2 & \cdots & \sigma_n\cdot e\cdot g_m & \sigma_n\cdot e(\chi^G)\\
\hline
\varepsilon_C(H,K)\cdot g_1 & \varepsilon_C(H,K)\cdot g_2 & \cdots & \varepsilon_C(H,K)\cdot g_m & *
\end{array}$$

We can compute the total sum $*$ by adding the elements of the last column or the elements of the last row:
\begin{eqnarray}\label{*} *=\sum_{i=1}^n \sigma_i \cdot e(\chi^G) = \sum_{j=1}^m \varepsilon_C(H,K)\cdot g_j .\end{eqnarray}

In the first sum of (\ref{*}) the elements to add are the elements of the $\Gal(F(\chi)/F)$-orbit of $e(\chi^G)$, each of them repeated $[F(\chi):F(\chi^G)]$ times. Using (\ref{yamada}), one has
\begin{eqnarray}\label{sum1} * = [F(\chi):F(\chi^G)] e_{F}(\chi^G). \end{eqnarray}
Similarly the second sum of (\ref{*}) adds up the elements of the $G$-orbit of $\varepsilon_C(H,K)$ (by equation (\ref{eq2})), each of them repeated $[\Cen_G(\varepsilon_C(H,K)):H]$ times. Therefore
\begin{eqnarray}\label{sum2} * = [\Cen_G(\varepsilon_C(H,K)):H]e_C(G,H,K).\end{eqnarray}
The theorem follows by comparing (\ref{sum1}) and (\ref{sum2}).
 \end{proof}

So for each Shoda pair $(H,K)$ of $G$, number field $F$ and each $C\in \mathcal{C}_F(H/K)$, there exists a unique $\alpha\in \Q$ such that $\alpha e_C(G,H,K)$ is a primitive central idempotent of $F G$.

\begin{lemma}\label{lemma} Let $F$ be a number field. Let $K\unlhd H\leq G$ be such that $H/K$ is cyclic of order $k$ and $C\in\mathcal{C}_F(H/K)$.
\begin{enumerate}
\item For every $g\in G$, the following statements are equivalent:
\begin{enumerate}
\item\label{s1} $g\in K$,
\item\label{s2} $g\varepsilon_C(H,K)=\varepsilon_C(H,K)$,
\item\label{s3} $\suma{g}\varepsilon_C(H,K)=\varepsilon_C(H,K)$.
\end{enumerate}
\item If $H\unlhd N_G(K)$, then $\Cen_G(\varepsilon_C(H,K))=E_F(G,H/K)$.
\end{enumerate}
\end{lemma}
\begin{proof}
1. The fact that (\ref{s1}) implies (\ref{s2}) follows from the easy observation that $g\suma{K}=\suma{K}$ when $g\in K$. The equivalence between (\ref{s2}) and (\ref{s3}) follows by comparing the coefficients. It remains to prove that $g\in K$ when $g\varepsilon_C(H,K)=\varepsilon_C(H,K)$. Assume that $g\varepsilon_C(H,K)=\varepsilon_C(H,K)$. Because of Corollary~\ref{cyclicquotient}, $\varepsilon_C(H,K)$ is a primitive central idempotent in $FH$ and hence non-zero. Therefore, the support of $g\varepsilon_C(H,K)$ is a non-empty set in $H$ and $g\in H$. Hence we can write $g=xh^t$ for some $x\in K$ and $0\leq t\le k$. Now 
$$k\inv\suma{K}\sum_{i=0}^{k-1} \sum_{\psi\in C} \psi(h^iK)h^{t-i}=g\varepsilon_C(H,K)=\varepsilon_C(H,K)=k\inv\suma{K}\sum_{i=0}^{k-1} \sum_{\psi\in C} \psi(h^iK)h^{-i}$$ and hence 
$$\tr_{F(\zeta_k)/F}((\chi(h^{t}K)-1)\chi(h^iK))=\sum_{\psi\in C}(\psi(h^{t}K)-1)\psi(h^iK)=0$$ for every $0\leq i\leq k-1$ and some $\chi\in C$. Since $\chi$ is faithful, its image generates $F(\chi)=F(\zeta_k)$ as $F$-vector space. Since $ \tr_{F(\zeta_k)/F}:F(\zeta_k)\rightarrow F$ is $F$-linear and surjective, we deduce that $\chi(h^tK)=1$ and hence $k$ divides $t$. Therefore $t=0$ and $g\in K$.

2. When $H\unlhd N_G(K)$, clearly $E_F(G,H/K)\subseteq N_G(H)\cap N_G(K)\subseteq N_G(K)$. Furthermore, if $g\in \Cen_G(\varepsilon_C(H,K))$, then for each $x\in K$ 
$$g\inv xg\varepsilon_C(H,K)=g\inv x\varepsilon_C(H,K)g=\varepsilon_C(H,K),$$
by 1. Therefore $g\inv xg\in K$ and $\Cen_G(\varepsilon_C(H,K))\subseteq N_G(K)$. Take $g\in N_G(K)$. By Corollary~\ref{cyclicquotient}, $\varepsilon_C(H,K)$ and $\varepsilon_{C^g}(H,K)$ are two primitive central idempotents of $FH$ and they are equal if and only if $C=C^g$ (i.e. if $g\in E_F(G,H/K)$). By equation (\ref{eq2}), $\varepsilon_C(H,K)^g=\varepsilon_{C^g}(H,K)$. Hence $g\in \Cen_G(\varepsilon_C(H,K))$ if and only if $\varepsilon_C(H,K)^g=\varepsilon_C(H,K)$, if and only if $g\in E_F(G,H/K)$.
 \end{proof}

We recall from \cite{Jespers2003,Olivieri2004} the following definitions. If $N\unlhd G$, then 
\begin{eqnarray*}
\varepsilon(G,N) = \left\{ \begin{array}{ll} \widehat{N} & \mbox{if } N=G \\
\prod_{M/N\in\mathcal{M}(G/N)} (\widehat{N}-\widehat{M}) = \widehat{N} \prod_{M/N\in\mathcal{M}(G/N)} (1-\widehat{M})& \mbox{if } N\neq G \end{array} \right.,
\end{eqnarray*} 
where $\mathcal{M}(G/N)$ is the set of all minimal normal non-trivial subgroups of $G/N$. If $K\unlhd H\leq G$, then $e(G,H,K)$ denotes the sum of all $G$-conjugates of $\varepsilon(H,K)$; i.e. $e(G,H,K)=\sum_{t\in T} \varepsilon(H,K)^t$ where $T$ is a right transversal of $\Cen_G(\varepsilon_C(H,K))$ in $G$.

Clearly elements of $\Q G$ can also be seen as elements in $FG$. The following lemma tells how $\varepsilon(G,N)$ and $e(G,H,K)$ can be written as a sum of elements in $FG$.

\begin{lemma}\label{lemma2}
Let $F$ be a number field.
\begin{enumerate}
 \item Let $N\unlhd G$ be such that $G/N$ is cyclic, then $$\varepsilon(G,N)=\sum_{C\in\mathcal{C}_F(G/N)}\varepsilon_C(G,N).$$\\
\item Let $K\unlhd H\unlhd N_G(K)$ be such that $H/K$ is cyclic and let $R$ be a set of representatives of the action of $N_G(K)$ on $\mathcal{C}_F(H/K)$. Then 
$$e(G,H,K)=\sum_{C\in R}e_C(G,H,K).$$
\end{enumerate}
\end{lemma}
\begin{proof}
 1. Both $\varepsilon(G,N)$ and $\varepsilon_C(G,N)$ belong to $FG\suma{N}$ for every $C\in\mathcal{C}_F(H/K)$. By factoring out $N$ and using the isomorphism $FG\suma{N}\simeq F(G/N)$, we may assume without loss of generality that $N=1$ and $G$ is cyclic. By Proposition~\ref{finiteabelian}, every primitive central idempotent of $FG$ is of the form $\varepsilon_C(G,H)$ with $H\unlhd G$ and $C\in \mathcal{C}_F(G/H)$. Therefore $\varepsilon(G,1)$ is the sum of some elements $\varepsilon_C(G,H)$ and it is enough to prove that if $H\unlhd G$ and $C\in\mathcal{C}_F(G/H)$, then $\varepsilon(G,1)\varepsilon_C(G,H)\neq 0$ if and only if $H=1$.

If $C\in \mathcal{C}_F(G)$ and $1\neq x\in G$, then $(1-\suma{x})\varepsilon_C(G,1)\neq 0$ by Lemma~\ref{lemma}. Since $\varepsilon_C(G,1)$ is a primitive central idempotent, we have that $\varepsilon_C(G,1)=(1-\suma{x})\varepsilon_C(G,1)$. Since $\varepsilon(G,1)$ is the product of elements of the form $1-\suma{x}$ with $1\neq x\in G$, it follows that $\varepsilon(G,1)\varepsilon_C(G,1)=\varepsilon_C(G,1)\neq 0$. Conversely, if $1\neq H\leq G$, then there exists a $h\in H$ such that $\GEN{h}$ is a minimal non-trivial subgroup of $G$. Hence $\varepsilon(G,1)\varepsilon_C(G,H)=\varepsilon(G,1)(1-\suma{h})\varepsilon_C(G,H)=0$, by Lemma~\ref{lemma}. This finishes part 1 of the proof.

2. Let $N=N_G(H)\cap N_G(K)=N_G(K)$, $E=E_F(G,H/K)$, $T_N$ be a right transversal of $N$ in $G$, $T_E$ be a right transversal of $E$ in $N$. Then $\{hg\mid h\in T_E, g\in T_N\}$ is a right transversal of $E$ in $G$. By \cite[Proposition 3,3]{Olivieri2004}, we know that $N=\Cen_G(\varepsilon(H,K))$. Hence $e(G,H,K)=\sum_{g\in T_N} \varepsilon(H,K)^g$. Clearly, $\mathcal{C}_F(H/K)$ is the disjoint union of the sets $\{C^t\mid t\in T_E\}$ for $C$ running on $R$. Therefore,
\begin{eqnarray*}
 e(G,H,K) &=& \sum_{g\in T_N} \varepsilon(H,K)^g \\
&=& \sum_{g\in T_N}\sum_{C\in\mathcal{C}_F(H/K)} \varepsilon_C(H,K)^g\\
&=& \sum_{g\in T_N}\sum_{C\in R}\sum_{h\in T_E} \varepsilon_{C^h}(H,K)^g\\
&=& \sum_{C\in R}\sum_{g\in T_N}\sum_{h\in T_E} \varepsilon_{C}(H,K)^{hg}\\
&=& \sum_{C\in R} e_C(G,H,K).
\end{eqnarray*}
 \end{proof}

\section{The structure of the Wedderburn decomposition for strongly monomial groups}

\begin{definition}
A strong Shoda pair of $G$ is a pair $(H,K)$ of subgroups of $G$ satisfying the following conditions:
\begin{enumerate}
\item \label{SS1} $K\leq H \unlhd N_G(K)$,
\item \label{SS2} $H/K$ is cyclic and a maximal abelian subgroup of $N_G(K)/K$,
\item \label{SS3} for every $g\in G\setminus N_G(K)$, $\varepsilon(H,K)\varepsilon(H,K)^g=0$.
\end{enumerate}
\end{definition}

Let $\chi$ be an irreducible character of $G$. The character $\chi$ is said to be strongly monomial if there is a strong Shoda pair $(H,K)$ of $G$ and a linear character $\psi$ of $H$ with kernel $K$ such that $\chi=\psi^G$, the induced character on $G$. The group $G$ is strongly monomial if every irreducible character of $G$ is strongly monomial.

In the main theorem of this section we describe the simple components of the group algebra $FG$ provided by strong Shoda pairs. We show that a strong Shoda pair $(H,K)$ of $G$ that determines a primitive central idempotent $e(G,H,K)$ in $\Q G$, will also determine a primitive central idempotent $e_C(G,H,K)$ in $FG$ for $C\in\mathcal{C}_F(G/N)$. Lemma~\ref{lemma2} shows how $e(G,H,K)$ splits into a sum of primitive central idempotents of $FG$.

\begin{theorem}\label{main}
Let $G$ be a finite group and $F$ be a number field.
\begin{enumerate}
 \item Let $(H,K)$ be a strong Shoda pair of $G$ and $C\in\mathcal{C}_F(H/K)$. Let $[H:K]=k$ and $E=E_F(G,H/K)$. Then $e_C(G,H,K)$ is a primitive central idempotent of $FG$ and 
$$FGe_C(G,H,K)\simeq M_{[G:E]}\left(F\left(\zeta_{k}\right)*_{\tau}^{\sigma}E/H\right),$$ where $\sigma$ and $\tau$ are defined as follows. Let $\phi:E/H\rightarrow E/K$ be a left inverse of the projection $E/K\rightarrow E/H$. Then
\begin{eqnarray*}
\sigma_{gH}(\zeta_k) &=& \zeta_k^i, \mbox{ if } yK^{\phi(gH)}=y^iK ,\\
\tau(gH,g'H) &=& \zeta_k^j, \mbox{ if }  \phi(gg'H)^{-1}\phi(gH)\phi(g'H)=y^jK,
\end{eqnarray*}
for $gH,g'H\in E/H$ and integers $i$ and $j$.
\item Let $X$ be a set of strong Shoda pairs of $G$. If every primitive central idempotent of $\Q G$ is of the form $e(G,H,K)$ for $(H,K)\in X$, then every primitive central idempotent of $FG$ is of the form $e_C(G,H,K)$ for $(H,K)\in X$ and $C\in\mathcal{C}_F(H/K)$.
\end{enumerate}
\end{theorem}
\begin{proof}
1. By Lemma~\ref{lemma}, $E=\Cen_G(\varepsilon_C(H,K))$ and by \cite[Proposition 3,3]{Olivieri2004}, $\Cen_G(\varepsilon(H,K)) = N_G(K)$. Let $T$ be a right transversal of $E$ in $G$, then $e_C(G,H,K)=\sum_{g\in T}\varepsilon_C(H,K)^g$.

In order to prove that $e_C(G,H,K)$ is an idempotent, it is enough to show that the $G$-conjugates of $\varepsilon_C(H,K)$ are orthogonal. For this we show that, if $g\in G\setminus E$, then $\varepsilon_C(H,K)\varepsilon_C(H,K)^g=0$. By Lemma~\ref{lemma2}, $$\varepsilon_C(H,K)\varepsilon_C(H,K)^g=\varepsilon_C(H,K)\varepsilon(H,K)\varepsilon(H,K)^g\varepsilon_C(H,K)^g,$$ which is zero by the definition of a strong Shoda pair whenever $g\notin N_G(K)$. If $g\in N_G(K)\setminus E$, then $\varepsilon_C(H,K)^g=\varepsilon_{C^g}(H,K)\neq\varepsilon_C(H,K)$ are two different primitive central idempotents of $FH$ by Corollary~\ref{cyclicquotient}. Hence $\varepsilon_C(H,K)\varepsilon_C(H,K)^g=0$. 

By Corollary~\ref{cyclicquotient}, $FH\varepsilon_C(H,K)\simeq F(\zeta_k)$. Also, $FE\varepsilon_C(H,K) =FH\varepsilon_C(H,K) *_{\tau}^{\sigma} E/H$ is a crossed product with homogeneous basis $\phi(E/H)$, where $\phi:E/H\rightarrow E/K$ is a left inverse of the projection $E/K\rightarrow E/H$. The action $\sigma$ and twisting $\tau$ are given by
\begin{eqnarray*}
&& \sigma:E/H\rightarrow \Aut(FH\varepsilon_C(H,K)): gH\mapsto (\alpha\varepsilon_C(H,K)\mapsto \phi(gH)\inv \alpha \varepsilon_C(H,K) \phi(gH)),\\
&& \tau: E/H\times E/H\rightarrow\U(FH\varepsilon_C(H,K)):(gH,g'H)\mapsto \phi(gg'H)\inv \phi(gH)\phi(g'H).
\end{eqnarray*}
Clearly, the isomorphism $FH\varepsilon_C(H,K)\simeq F(\zeta_k)$ extends to an $E/H$-graded isomorphism $$FE\varepsilon_C(H,K)=FH\varepsilon_C(H,K) *_{\tau}^{\sigma} E/H\simeq F(\zeta_k)*_{\tau'}^{\sigma'} E/H.$$ Since $H/K$ is maximal abelian in $N/K$ and hence also in $E/K$, the action $\sigma'$ is faithful and $FE\varepsilon_C(H,K)$ is simple (\cite[Theorem 29.6]{Reiner1975}). 

If $g\in G$, then the map $x\mapsto xg$ defines an isomorphism between the $FG$-modules $FG\varepsilon_C(H,K)$ and $FG\varepsilon_C(H,K)^g$. Therefore, $_{FG}FGe_C(G,H,K)=\bigoplus_{t\in T}FG\varepsilon_C(H,K)^t\simeq (FG\varepsilon_C(H,K))^{[G:E]}$. Moreover, $\varepsilon_C(H,K)FG\varepsilon_C(H,K)=\bigoplus_{t\in T}\varepsilon_C(H,K)FEt\varepsilon_C(H,K)=\bigoplus_{t\in T}FE\varepsilon_C(H,K)t\varepsilon_C(H,K)=FE\varepsilon_C(H,K)$, because $\varepsilon_C(H,K)$ is central in $FE$ and $\varepsilon_C(H,K)^t\varepsilon_C(H,K)=0$ for all $t\in G\setminus E$. Thus
\begin{eqnarray*}
 FGe_C(G,H,K)&\simeq& \End_{FG}(FGe_C(G,H,K))\simeq M_{[G:E]}(\End_{FG}(FG\varepsilon_C(H,K)))\\
&\simeq& M_{[G:E]}(\varepsilon_C(H,K)FG\varepsilon_C(H,K))=M_{[G:E]}(FE\varepsilon_C(H,K))\\
&\simeq& M_{[G:E]}(F(\zeta_k)*_{\tau'}^{\sigma'} E/H).
\end{eqnarray*}

2. By assumption there is a set $Y\subseteq X$ such that $\{e(G,H,K)\mid (H,K)\in Y\}$ is the set of primitive central idempotents of $\Q G$. Hence, by Lemma~\ref{lemma2}, $$1=\sum_{(H,K)\in Y}e(G,H,K)=\sum_{(H,K)\in Y}\sum_{C\in R_{(H,K)}}e_C(G,H,K),$$ for $R_{(H,K)}$ a set of representatives of the action of $N_G(K)$ on $\mathcal{C}_F(H/K)$. Furthermore, $e_C(G,H,K)$ are primitive central idempotents of $FG$ by (1). Hence $\{e_C(G,H,K)\mid (H,K)\in Y, C\in R_{(H,K)}\}$ is the set of primitive central idempotents of $FG$.
 \end{proof}

Applying the result of \cite[Theorem 4.4]{Olivieri2004}, we get the following result.

\begin{corollary}
 If $G$ is a finite strongly monomial group (e.g. a finite abelian-by-supersolvable group) and $F$ a number field. Then every primitive central idempotent if $FG$ is of the form $e_C(G,H,K)$ for a strong Shoda pair $(H,K)$ of $G$ and $C\in\mathcal{C}_F(H/K)$. Furthermore, for every strong Shoda pair $(H,K)$ of $G$ and every $C\in \mathcal{C}_F(H/K)$, $$FGe_C(G,H,K)\simeq M_{[G:E]}\left(F\left(\zeta_{[H:K]}\right)*_{\tau}^{\sigma}E/H\right),$$ where $\sigma$ and $\tau$ are defined as above and $E=E_F(G,H/K)$. 
\end{corollary}

Applying the result of \cite[Theorem 4.7]{Olivieri2004}, we get the following result.

\begin{corollary}\label{metabelian}
  If $G$ is a finite metabelian group, $A$ a maximal abelian subgroup of $G$ containing $G'$ and $F$ a number field. Then every primitive central idempotent of $FG$ is of the form $e_C(G,H,K)$ for a pair $(H,K)$ of subgroups of $G$ satisfying the following conditions: 
\begin{enumerate}
\item \label{metabelian1}$H$ is a maximal element in the set $\{B\leq G \mid A\leq B \mbox{ and } B'\leq K\leq B\}$;
\item \label{metabelian2}$H/K$ is cyclic;
\end{enumerate}
 and $C\in\mathcal{C}_F(H/K)$. Furthermore, for every pair $(H,K)$ of subgroups of $G$ satisfying (\ref{metabelian1}) and (\ref{metabelian2}) and every $C\in \mathcal{C}_F(H/K)$, $$FGe_C(G,H,K)\simeq M_{[G:E]}\left(F\left(\zeta_{[H:K]}\right)*_{\tau}^{\sigma}E/H\right),$$ where $\sigma$ and $\tau$ are defined as above and $E=E_F(G,H/K)$. 
\end{corollary}

\begin{remark}
The description in Theorem~\ref{main} agrees with the one given in \cite[Proposition 1]{2007Olteanu} and provides an alternative formula for the simple components in a group algebra over a number field for strongly monomial groups. Our formula has the advantage that it avoids induced characters and the computation of some fields of character values. Moreover, it contains more details about the crossed product. 
\end{remark}

\section{On the number of simple components}

In \cite{Ferraz2007}, Ferraz and Polcino Milies showed some conditions on when the number of simple components of finite semisimple abelian group algebras $\F G$ reaches the lower bound, i.e. when it coincide with the number of simple components of $\Q G$. We study the number of simple components of $F G$ for both number fields and finite fields (with $\textnormal{char}(F)\nmid |G|$) and investigate when this number is minimal for strongly monomial groups.

From now on we assume $G$ to be a finite strongly monomial group. We say that two strong Shoda pairs $(H_1,K_1)$ and $(H_2,K_2)$ of $G$ are equivalent when $e(G,H_1,K_1)=e(G,H_2,K_2)$, or equivalently if there is a $g\in G$ such that $H_1^g\cap K_2=K_1^g\cap H_2$ \cite{Olivieri2006}. We call a set $X$ of strong Shoda pairs of $G$ a complete set of representatives if $X$ contains at most one strong Shoda pair in each equivalence class and $\sum_{(H,K)\in X}e(G,H,K)=1$. In this way, we guarantee that the set $X$ associates to a complete and non-redundant set of primitive central idempotents for a strongly monomial group $G$.

For $\Q G$ the number of simple components clearly coincide with the number of equivalence classes on strong Shoda pairs. Because of Lemma~\ref{lemma2}, we know that for each strong Shoda pair $(H,K)$, $e(G,H,K)=\sum_{C\in R}e_C(G,H,K)$ where $R$ is a set of representatives of the action of $N_G(K)$ on $\mathcal{C}_F(H/K)$. Hence $e(G,H,K)$ decomposes in a certain number of primitive central idempotents, which is exactly the number of orbits of the action of $N_G(K)$ on $\mathcal{C}_F(H/K)$. This number of orbits is exactly 
$$\frac{|\mathcal{C}_F(H/K)|}{|\Orb_{N_G(K)}(C)|}=\frac{\phi([H:K])}{|I_{[H:K]}(F)|}\frac{|E_F(G,H/K)|}{|N_G(K)|} $$ for any $C\in \mathcal{C}_F(H/K)$. 

We have the following embeddings of fields:
\begin{center}
\begin{tikzpicture}[node distance=2cm]
\node(FzEH)                     	{$F(\zeta_{[H:K]})^{E_F(G,H/K)/H}$};
\node(Fz) 	[above left of=FzEH] 	{$F(\zeta_{[H:K]})$};
\node(F)	[below right of=FzEH]	{$F$};
\node(QzNH)	[below left of=FzEH]	{$\Q(\zeta_{[H:K]})^{N_G(K)/H}$};
\node(Qz)	[above left of=QzNH]	{$\Q(\zeta_{[H:K]})$};
\node(cap)	[below left of=F]	{$\Q(\zeta_{[H:K]})^{N_G(K)/H}\cap F$};
\node(Q)	[below right of=cap]		{$\Q$};

\draw(FzEH)	-- (Fz);
\draw(F)	-- (FzEH);
\draw(QzNH)	-- (FzEH);
\draw(QzNH)	-- (Qz);
\draw(QzNH)	-- (cap);
\draw(F)	-- (cap);
\draw(Q)	-- (cap);
\draw(Fz)	-- (Qz);
\end{tikzpicture}
\end{center}
Here we have used the notation $L^A$ to denote the subfield of the field $L$, fixed by the action of a group $A$ on $L$.
From this diagram it follows that 
\begin{eqnarray*}
 [\Q(\zeta_{[H:K]})^{N_G(K)/H}\cap F:\Q]
&=& \frac{[\Q(\zeta_{[H:K]})^{N_G(K)/H}:\Q]}{[\Q(\zeta_{[H:K]})^{N_G(K)/H}:\Q(\zeta_{[H:K]})^{N_G(K)/H}\cap F]}\\
&=& \frac{\phi([H:K])/[N_G(K):H]}{[F(\zeta_{[H:K]})^{E_F(G,H/K)/H}:F]}\\
&=& \frac{\phi([H:K])/[N_G(K):H]}{|I_{[H:K]}(F)|/[E_F(G,H/K):H]}.
\end{eqnarray*} 

Hence the following proposition follows.

\begin{proposition}\label{number}
 Let $G$ be a finite strongly monomial group and $F$ a number field. Let $X$ be a complete set of representatives of strong Shoda pairs of $G$. Then the number of simple components of $FG$ equals 
$$\sum_{(H,K)\in X}\frac{\phi([H:K])}{|I_{[H:K]}(F)|}\frac{|E_F(G,H/K)|}{|N_G(K)|}=\sum_{(H,K)\in X}[\Q(\zeta_{[H:K]})^{N_G(K)/H}\cap F:\Q].$$
\end{proposition}

Clearly the number of simple components of $\Q G$ is a lower bound for the number of components of $F G$. We say that the number of components is minimal when it reaches this lower bound. For a strong Shoda pair $(H,K)$ with $H/K=\GEN{hK}$, let $\U_{(H,K)}=\{r\in \U(\Z/[H:K]\Z)\mid h^gh^{-r}\in K \mbox{ for some }g\in N_G(K)\}$.

\begin{proposition}\label{minimal}
 Let $G$ be a finite strongly monomial group and $F$ a number field. Let $X$ be a complete set of representatives of strong Shoda pairs of $G$. 
The following statements are equivalent:
\begin{enumerate}
 \item The number of simple components of $F G$ is minimal;
 \item For each pair $(H,K)\in X$: $\U(\Z/[H:K]\Z)=\GEN{I_{[H:K]}(F), \U_{(H,K)}}$;
 \item For each pair $(H,K)\in X$: $\phi([H:K])=|I_{[H:K]}(F)|\frac{|N_G(K)|}{|E_F(G,H/K)|}$ for any $C\in \mathcal{C}_F(H/K)$;
 \item For each pair $(H,K)\in X$: $[\Q(\zeta_{[H:K]})^{N_G(K)/H}\cap F:\Q]=1$.
\end{enumerate}
\end{proposition}
\begin{proof}
For each $(H,K)\in X$, fix $h\in H$ such that $H/K=\GEN{hK}$ and fix a faithful irreducible character $\chi$ of $H/K$. Then any faithful linear character of $H/K$ is of the form $\chi^i$ for some $(i,[H:K])=1$.

The number of simple components of $F G$ is minimal if and only if for each strong Shoda pair $(H,K)\in X$ the idempotent $e(G,H,K)$ is already primitive in $F G$. This happens whenever the action of $N_G(K)$ on $\mathcal{C}_F(H/K)$ is transitive for any pair $(H,K)\in X$. This is equivalent to the existence of $\sigma_i\in I_{[H:K]}(F)$ and $g_i\in N_G(K)$ for any $i$ with $(i,[H:K])=1$ such that $(\chi^{g_i})^{\sigma_i}=\chi^i$. The above means that $\U(\Z/[H:K]\Z)=\GEN{I_{[H:K]}(F), \U_{(H,K)}}$, which proves the equivalence between the first two statements. The equivalence with the third and fourth statement follows from Proposition~\ref{number}.
 \end{proof}

The following sufficient condition follows immediately.
\begin{corollary}\label{sufficient}
Let $G$ be a finite strongly monomial group and $F$ a number field. Let $X$ be a complete set of representatives of strong Shoda pairs of $G$. 
If $\phi([H:K])=|I_{[H:K]}(F)|$ for each $(H,K)\in X$, then the number of simple components of $F G$ is minimal.
\end{corollary}
\begin{proof}
 Let $(H,K)$ be a strong Shoda pair of $G$ such that $\phi([H:K])=|I_{[H:K]}(F)|$ and denote $k=[H:K]$ and $hK$ a generator of $H/K$. Then $\Gal(F(\zeta_k)/F)=\Gal(\Q(\zeta_k)/\Q)$. Since $g^{-1}hg \in H$ for all $g\in N_G(K)$ and the order of $g^{-1}hgK$ is the same as the order of $hK$, it follows that $g^{-1}hgK=h^iK$ for some $i\in \U(\Z/k\Z)=I_k(F)$ and $E_F(G,H/K)=N_G(K)$. Now the statement follows from Proposition~\ref{minimal}.
 \end{proof}

However, the previous condition is not necessary. 
\begin{example}\label{D6}
 Let $G=D_6=\GEN{a,b\mid a^ 3=1,b^2=1,bab=a^2}$ and $F=\Q(\zeta_3)$. Take $H=\GEN{a}$ and $K=1$, then $(H,K)$ is a strong Shoda pair of $G$ such that $N_G(K)=G$, $I_{[H:K]}(F)=1$ and $E_F(G,H/K)=\GEN{a}$. Clearly, $|I_{[H:K]}(F)|\neq \phi([H:K])=2$ but still $\frac{\phi([H:K])}{|I_{[H:K]}(F)|}\frac{|E_F(G,H/K)|}{|N_G(K)|}=1$.
\end{example}

Also, it is clear that $E_F(G,H/K)=N_G(K)$ provided $\phi([H:K])=|I_{[H:K]}(F)|$, however the opposite is not valid. 
\begin{example}
Let $G=C_3\times Q_8=\GEN{a,x,y\mid a^3=1, x^4=1, y^2=1, ax=xa, ay=ya, yxy=x^3}$ and $F=\Q(\zeta_3)$. Then $(\GEN{ax},1)$ is a strong Shoda pair of $G$ and $I_{12}(F)=\{1,7\}\neq \U(\Z/12\Z)$ and still $E_F(G,\GEN{ax})=G=N_G(1)$. 
\end{example}

\begin{lemma}\label{galois2}
 Let $F$ be a number field and $n$ and $m$ integers such that $n$ divides $m$. If $\phi(m)=|I_{m}(F)|$, then $\phi(n)=|I_{n}(F)|$.
\end{lemma}
\begin{proof}
Because of the Chinese Remainder Theorem, there is a surjective mapping from $\U(\Z/m\Z)$ to $\U(\Z/n\Z)$. Assume that $\phi(m)=|I_{m}(F)|$.
We have the following commutative diagram of sets

$$\SelectTips{cm}{}
\xymatrix{
\Gal(F(\zeta_n)/F) \hspace{2mm}\ar@{^{(}->}[r] & \hspace{1mm} \U(\Z/n\Z) \\ 
\Gal(F(\zeta_m)/F) \ar@{->>}[u] \ar@{=}[r] & \U(\Z/m\Z). \ar@{->>}[u]
}$$
It follows that $|I_n(F)|=|\Gal(F(\zeta_n)/F)|=|\U(\Z/n\Z)|=\phi(n)$.
 \end{proof}

\begin{lemma}\label{galois}
 Let $k$ be an integer and $F$ a number field. Then $|I_k(F)|=\phi(k)$ if and only $F\cap \Q(\zeta_{k}) = \Q$.
\end{lemma}
\begin{proof}
Because of the Fundamental Theorem of Galois, $\Gal(F(\zeta_k)/F)\simeq \Gal(\Q(\zeta_k)/F\cap \Q(\zeta_k))$.
Assume that $|I_k(F)|=\phi(k)$, then $I_k(F)=\U(\Z/k\Z)$ or equivalently, $\Gal(\Q(\zeta_k)/F\cap \Q(\zeta_k))\simeq\Gal(F(\zeta_k)/F)\simeq\Gal(\Q(\zeta_k)/\Q)$. This is equivalent with $F\cap \Q(\zeta_{k}) = \Q$.
 \end{proof}

\begin{corollary}
  Let $G$ be a finite strongly monomial group of exponent $e$ and $F$ a number field. 
If $\phi(e)=|I_{e}(F)|$ then the number of simple components of $F G$ is minimal.
\end{corollary}
\begin{proof}
 For each strong Shoda pair $(H,K)$ of $G$, the group $H/K$ is cyclic. Therefore, the index $[H:K]$ is a divisor of the exponent of $G$. 
Because of Lemma~\ref{galois2}, $|I_n(F)|=\phi(n)$. Hence the result follows from Corollary~\ref{sufficient}.
 \end{proof}

When we have more details in the structure of the group itself, we can give general conditions for the number of components being minimal.

%

\begin{corollary}\label{abelian}
 Let $G$ be a finite abelian group of exponent $e$ and $F$ a number field. The number of simple components of $F G$ is minimal if and only if $\U(\Z/e\Z)=I_e(F)$.
\end{corollary}
\begin{proof}
 When $G$ is abelian, any strong Shoda pair is of the form $(G,N)$ with $G/N$ cyclic. Since $G$ is abelian, clearly $\U_{(G,N)}=1$ for any strong Shoda pair. 

By Proposition~\ref{minimal}, it is sufficient to check that $\U(\Z/e\Z)=I_e(F)$ if and only if $\U(\Z/[G:N]\Z)=I_{[G:N]}(F)$ for each $N$ such that $G/N$ is cyclic. 

Suppose that $\U(\Z/e\Z)=I_e(F)$. Let $(G,N)$ be a strong Shoda pair. Then $[G:N]$ divides $e$. By Lemma~\ref{galois2}, it follows that $\U(\Z/[G:N]\Z)=I_{[G:N]}(F)$.

On the other hand, suppose that $\U(\Z/[G:N]\Z)=I_{[G:N]}(F)$ for each $N$ such that $G/N$ is cyclic. Since $G$ is abelian and $e$ is the exponent of $G$, there exists a subgroup $N$ of $G$ such that $G/N$ is cyclic of order $e$. It follows that $\U(\Z/e\Z)=I_e(F)$. 
 \end{proof}

Corollary~\ref{metabelian} allows one to easily compute the primitive central idempotents of the group algebra of a finite metacyclic group. Every finite metacyclic group $G$ has a presentation of the form $$G=\GEN{a}_m\rtimes_k \GEN{b}_n = \GEN{a,b\mid a^m=1,\ b^n=a^t,\ a^b=a^r},$$ where $m,n,t,r$ are integers satisfying the conditions $r^n \equiv 1 \mod m$, $m\mid t(r-1)$ and $o_m(r)=\frac{n}{k}$. For every $d\mid \frac{n}{k}$, let $G_d=\langle a,b^d\rangle$. 
The primitive central idempotents of $F G$ are the elements of the form $e_C(G,G_d,K)$ where $C\in \mathcal{C}_F(G_d/K)$, $d$ is a divisor of $\frac{n}{k}$ and $K$ is a subgroup of $G_d$ satisfying the following conditions:
\begin{enumerate}
\item $d=\min\{x\mid \frac{n}{k} : a^{r^x-1}\in K\}$,
\item $G_d/K$ is cyclic.
\end{enumerate}

%
%
%
%
%

This description allows us to classify when the number of simple components is minimal for some particular classes of metacyclic groups.

\begin{corollary}
 Let $F$ be a number field and $G=\GEN{a}_m\rtimes \GEN{b}_p$ with $p$ a prime which does not divide the integer $m$. Assume that $b^{-1}ab=a^r$ with $(r-1,m)=1$. Then the number of simple components of $FG$ is minimal if and only if $\Q(\zeta_p)\cap F=\Q$ (or equivalently $|I_p(F)|=p-1$) and $\Q(\zeta_m)^{\GEN{b}}\cap F=\Q$. Here $\GEN{b}$ acts on $\Q(\zeta_m)$ by the rule $b\cdot \zeta_m=\zeta_m^r$ and we denote by $\Q(\zeta_m)^{\GEN{b}}$ the fixed subfield of $\Q(\zeta_m)$ under the action of $\GEN{b}$.
\end{corollary}
\begin{proof}
 The strong Shoda pairs of $G$ in a complete set of representatives are either of the form $(G,K)$ or $(\GEN{a},\GEN{a^l})$ for $l\mid m$. Here $a^{r-1}\in K$ and hence $a\in K$ and $K=G$ or $K=\GEN{a}$. 
Both $(G,\GEN{a})$ and $(\GEN{a},1)$ are always strong Shoda pairs of $G$ and all $[G:K]$ divide $[G:\GEN{a}]=p$ and all $[\GEN{a}:\GEN{a^l}]$ divide $[\GEN{a}:1]=m$.

By Proposition~\ref{minimal}, the number of simple components of $FG$ is minimal if and only if $[\Q(\zeta_{[G:K]})\cap F:\Q]=1=[\Q(\zeta_{[\GEN{a}:\GEN{a^l}]})^{\GEN{b}}\cap F:\Q]$. Since $$\Q\subseteq \Q(\zeta_p)\cap F$$ and $$\Q\subseteq \Q(\zeta_l)^{\GEN{b}}\cap F \subseteq \ldots \subseteq \Q(\zeta_m)^{\GEN{b}}\cap F,$$ the number of simple components of $FG$ is minimal if and only if $\Q(\zeta_p)\cap F=\Q$ (or equivalently $|I_p(F)|=\phi(p)=p-1$ by Lemma~\ref{galois}) and $\Q(\zeta_m)^{\GEN{b}}\cap F=\Q$.
 \end{proof}

\begin{corollary}\label{qn}
 Let $F$ be a number field and $G=\GEN{a}_q\rtimes_k \GEN{b}_n$ with $q$ a prime which does not divide the integer $n$. Assume that $b^{-1}ab=a^r$ with $o_q(r)=\frac{n}{k}$. Then the number of simple components of $FG$ is minimal if and only if $\Q(\zeta_n)\cap F=\Q$ (or equivalently $|I_n(F)|=\phi(n)$) and $\Q(\zeta_{qk})^{\GEN{b}}\cap F=\Q$. Here $\GEN{b}$ acts on $\Q(\zeta_{qk})$ by the rules $b\cdot \zeta_{q}=\zeta_q^r$ and $b\cdot \zeta_{k}=\zeta_k$ and we denote by $\Q(\zeta_q)^{\GEN{b}}$ the fixed subfield of $\Q(\zeta_q)$ under the action of $\GEN{b}$.
\end{corollary}
\begin{proof}
Consider the groups $G_d=\GEN{a,b^d}$ for $d$ a divisor of $\frac{n}{k}$. Assume that $d\neq 1$ and $d\neq \frac{n}{k}$ and $(G_d,K)$ is a strong Shoda pair of $G$. Then by the conditions, $a^{r-1}\notin K$ and thus $\GEN{a}\nsubseteq K$. However also $1\neq a^{r^d-1}\in K$ and therefore $a\in K$, a contradiction.

Therefore, the strong Shoda pairs of $G$ in a complete set of representatives are either of the form $(G,\GEN{a,b^l})$ with $l\mid n$ or $(\GEN{a,b^{\frac{n}{k}}},\GEN{b^{h\frac{n}{k}}})$ for $h\mid k$. 
Both $(G,\GEN{a})$ and $(\GEN{a,b^{\frac{n}{k}}},1)$ are always strong Shoda pairs of $G$ and all $[G:\GEN{a,b^l}]$ divide $[G:\GEN{a}]=n$ and all $[\GEN{a,b^{\frac{n}{k}}}:\GEN{b^{h\frac{n}{k}}}]$ divide $[\GEN{a,b^{\frac{n}{k}}}:1]=qk$.

Note that $Z(G)=\GEN{b^{n/k}}$. Therefore $N_G(\GEN{b^{h\frac{n}{k}}})=G$ and $G/\GEN{a,b^{n/k}}$ acts the same on $\GEN{a}$ as $\GEN{b}$ does.

By Proposition~\ref{minimal}, the number of simple components of $FG$ is minimal if and only if $[\Q(\zeta_{l})\cap F:\Q]=1=[\Q(\zeta_{qh})\cap F:\Q]$. Since $$\Q\subseteq \Q(\zeta_l)\cap F\subseteq \ldots \subseteq \Q(\zeta_n)\cap F$$ and $$\Q\subseteq \Q(\zeta_{qh})^{\GEN{b}}\cap F \subseteq \ldots \subseteq \Q(\zeta_{qk})^{\GEN{b}}\cap F,$$ the number of simple components of $FG$ is minimal if and only if $\Q(\zeta_n)\cap F=\Q$ (or equivalently $|I_n(F)|=\phi(n)$ by Lemma~\ref{galois}) and $\Q(\zeta_{qk})^{\GEN{b}}\cap F=\Q$.
 \end{proof}

\begin{corollary}\label{faithfulmetacyclic}
 Let $F$ be a number field and $G=\GEN{a}_{q^m}\rtimes_1 \GEN{b}_{p^n}$ with $p$ and $q$ different primes and $b^{-1}ab=a^r$ with $o_{q^m}(r)=p^n$. The number of simple components of $F G$ is minimal if and only if $\Q(\zeta_{p^n})\cap F=\Q$ (or equivalently $|I_{p^n}(F)|=p^{n-1}(p-1)$) and $\Q(\zeta_{q^m})^{\GEN{b}}\cap F=\Q$. Here $\GEN{b}$ acts on $\Q(\zeta_{q^m})$ by the rule $b\cdot \zeta_{q^m}=\zeta_{q^m}^r$ and we denote by $\Q(\zeta_{q^m})^{\GEN{b}}$ the fixed subfield of $\Q(\zeta_{q^m})$ under the action of $\GEN{b}$.
\end{corollary}
\begin{proof}
Let $\sigma$ be the automorphism of $\GEN{a}$ given by $\sigma(a)=a^r$. Then $\sigma$ has order $p^n$. As the kernel of the restriction map $\Aut(\GEN{a})\rightarrow \Aut\left(\GEN{a^{q^{m-1}}}\right)$ has order $q^{m-1}$ it intersects $\GEN{\sigma}$ trivially and therefore the restriction of $\sigma$ to $\GEN{a^{q^{m-1}}}$ also has order $p^n$. This implies that $q\equiv 1 \mod p^n$ and thus $q$ is odd. Therefore, $\Aut\left(\GEN{a^{q^j}}\right)(=\Gal(\Q(\zeta_{q^j})/\Q)=\U(\Z/q^j\Z))$ is cyclic for every $j=0,1,\dots,m$ and $\GEN{\sigma}$ is the unique subgroup of $\Aut(\GEN{a})$ of order $p^n$. So, for every $i=1,\dots,m$, the image of $r$ in $\Z/q^i \Z$ generates the unique subgroup of $\U(\Z/q^i \Z)$ of order $p^n$. In particular $r^{p^n}\equiv 1 \mod q^m$ and $r^{p^j}\not\equiv 1 \mod q$ for every $j=0,\ldots,n-1$. Therefore, $r\not\equiv 1 \mod q$ and hence $\GEN{a^{r-1}}=\GEN{a}$. Using the description of strong Shoda pairs of $G$ and using the arguments as in the proof of Corollary~\ref{qn}, the 
strong Shoda pairs of $G$ in a complete set of representatives are of two types:
\begin{enumerate}
\item $\left(G,\GEN{a,b^{p^i}}\right), \; i=0,\dots,n$,\\
\item $\left(\GEN{a},\GEN{a^{q^j}}\right), \; j=1,\dots,m$.
\end{enumerate}

Now, the number of simple components of $F G$ is minimal if and only if $[\Q(\zeta_{p^i})\cap F:\Q]=1=[\Q(\zeta_{q^j})^{\GEN{b}}\cap F:\Q]$ for all $i=0,\dots,n$ and $j=1,\dots,m$. Since $$\Q\subseteq \Q(\zeta_p)\cap F\subseteq \Q(\zeta_{p^2})\cap F\subseteq \ldots \subseteq \Q(\zeta_{p^n})\cap F$$ and 
$$\Q \subseteq \Q(\zeta_{q})^{\GEN{b}}\cap F\subseteq \Q(\zeta_{q^2})^{\GEN{b}}\cap F\subseteq \ldots\subseteq \Q(\zeta_{q^m})^{\GEN{b}}\cap F,$$ the above is equivalent with $\Q(\zeta_{p^n})\cap F=\Q$ and $\Q(\zeta_{q^m})^{\GEN{b}}\cap F=\Q$.
 \end{proof}

The conditions in the previous corollary are clearly fulfilled when $\Q(\zeta_{p^n})\cap F=\Q$ and $\Q(\zeta_{q^m})\cap F=\Q$ (or equivalently $|I_{p^n}(F)|=p^{n-1}(p-1)$ and $|I_{q^m}(F)|=q^{m-1}(q-1)$). However the second condition is not needed as we showed already in Example~\ref{D6}.
We give another example, when $\zeta_{q^m}$ might be in $F$.

\begin{example}
The conditions of Corollary~\ref{faithfulmetacyclic} hold for $F=\Q(\zeta_3)$ and $G=\GEN{a}_3\rtimes \GEN{b}_4$ with $b^{-1}ab=a^2$. Clearly $\Q(\zeta_{4})\cap F=\Q$. Since $\zeta_2+\zeta_3^2=-1$, also $\Q(\zeta_{3})^{\GEN{2}}\cap F=\Q$. 
\end{example}

\begin{remark}
There is a strong correspondence between the simple components in semisimple finite group algebras and simple components in group algebras over number fields. Let $\F_{q}$ be the finite field of order $q=p^n$ then $\F_q(\zeta_k)=\F_{q^{o_{k}(q)}}$ for an integer $k$ coprime to $p$. The Galois group $\Gal(\F_q(\zeta_k)/\F_q)$ is cyclic of order $o_k(p)$ and can be seen as a subgroup of $\U(\Z/k\Z)$. With the notations in this paper, for a cyclic group $G$, $\mathcal{C}_{\F_q}(G)$ agrees with the $q$-cyclotomic classes of $G$, denoted by $\mathcal{C}(G)$ in \cite{Broche2007}. 
Assume now that $G$ is a finite group of order coprime to $p$. 
Then $E_{\F_q}(G,H/K)$ agrees with the notion denoted by $E_G(H/K)$ in \cite{Broche2007} for a strong Shoda pair $(H,K)$ of $G$. In this way, the primitive central idempotents for strongly monomial groups are determined by strong Shoda pairs $(H,K)$ and $C\in \mathcal{C}_{\F_q}(H/K)$ \cite[Theorem 7]{Broche2007}. There is an analogue computation for how many primitive central idempotents are build by a fixed strong Shoda pair \cite[Lemma 6]{Broche2007}. In this way, we show that the number of simple components of $\F_q G$ is minimal (i.e. is the same as the number of components of $\Q G$) if and only if $$\frac{\phi([H:K])}{o_{[H:K]}(q)} \frac{|E_{\F_q}(G,H/K)|}{|N_G(K)|}=1$$ for all strong Shoda pairs $(H,K)$ in a complete set of representatives. Let $k=[H:K]$, one can compute that 
\begin{eqnarray*}
 \frac{\phi(k)}{o_{k}(q)} \frac{|E_{\F_q}(G,H/K)|}{|N_G(K)|}&=&[\F_{p^{\phi(k)}}:\F_p(\zeta_{k})][\F_p(\zeta_k)^{N_G(K)/H} \cap \F_q:\F_p]\\
&=& [\F_{p^{\phi(k)}}:\F_{p^{o_k(p)}}][\F_{p^{o_k(p)/[N_G(K):H]}} \cap \F_q:\F_p].
\end{eqnarray*}
This means that $\frac{\phi(k)}{o_{k}(q)} \frac{|E_{\F_q}(G,H/K)|}{|N_G(K)|}=1$ if and only if both $ [\F_{p^{\phi(k)}}:\F_{p^{o_k(p)}}]=1$ and $[\F_{p^{o_k(p)/[N_G(K):H]}} \cap \F_q:\F_p]=1$. Note that the former condition means that $o_k(p)=\phi(k)$.
Also an analogue of Corollary~\ref{sufficient} for the finite case follows, i.e if $\phi([H:K])=o_{[H:K]}(q)$ for each $(H,K)$ in a complete set of representatives of strong Shoda pairs of $G$, then the number of simple components of $\F_q G$ is minimal. As a consequence, we immediately see that if the number of simple components of $\F_qG$ is minimal, necessarily also the number of simple components of $\F_p G$ is minimal.

As an analogue of Corollary~\ref{abelian}, we find \cite[Theorem 2.2]{Ferraz2007} which says that the number of simple components of $\F_qG$ is minimal if and only if $\phi(e)=o_e(q)$ where $e$ is the exponent of $G$.

Similarly, one can obtain necessary and sufficient conditions on other groups for the minimality of the number of simple components over a fixed finite field $\F_q$. However, note that the condition $\phi([H:K])=o_{[H:K]}(p)$ for each $(H,K)$ in a complete set of representatives of strong Shoda pairs of $G$, already strongly restricts the possibilities for $G$ and its strong Shoda pairs. Indeed, it is well know that $\phi(k)=o_{k}(p)$ is only possible when $\U(\Z/k\Z)$ is cyclic. This happens only when $k=2,4,s^n$ or $2s^n$ for an odd prime $s$ and a positive integer $n$.
\end{remark}

\section{The rank of $\mathcal{Z}(\U(R G))$ for $R$ the ring of integers of $F$}

Let $F$ be a number field and $R$ its ring of integers. Let $G$ be a finite group. It is a classical problem to study the group of units of $RG$ and in particular to describe a finite set of elements that generate the unit group up to finite index. In this study the central units of $RG$,  $\mathcal{Z}(\U(R G))$ show up naturally \cite{JesParSeh1996,GJ,JesPar2011,2012JdROVG,2013JdROVG}. 
We give a formula for the rank of $\mathcal{Z}(\U(R G))$. Using Dirichlet's Unit Theorem one can prove that the rank of $\mathcal{Z}(\U(R G))$ is the difference between the number of simple components of $\R\otimes_{\Q} FG$ and the number of simple components of $F G$, see \cite[Theorem 3.5]{Ferraz2004}.

We first recall the following theorem.
\begin{theorem}\cite[Theorem 3.1]{2013JdROVG}\label{rank}
 Let $G$ be a finite strongly monomial group. Let $X$ be a complete set of representatives of strong Shoda pairs of $G$. Then the rank of $\mathcal{Z}(\U(\Z G))$ equals  $$\sum_{(H,K)\in X} \left(\frac{\phi([H:K])}{k_{(H,K)}[N:H]}-1\right),$$ where $h$ is such that $H=\GEN{h,K}$ and $$k_{(H,K)}=\left\{\begin{array}{ll} 1, & \mbox{if } hh^n\in K \mbox{ for some } n\in N_G(K); \\ 2, & \mbox{otherwise}.\end{array} \right.$$
\end{theorem}

Let us denote by $r_L(G)$ the number of simple components of $L G$ for a field $L$ and a finite group $G$.

\begin{lemma}\label{tensor}
 Let $F$ be a number field and $G$ a finite group. Then the number of simple components of $\R\otimes_{\Q} FG$ equals $r r_{\R}(G)+sr_{\C}(G)$, where $r$ is the number of real embeddings of $F$ and $2s$ is the number of complex embeddings of $F$.
\end{lemma}
\begin{proof}
 This follows from the equalities:
\begin{eqnarray*}
 \R\otimes_{\Q} FG &=& \R\otimes_{\Q} F \otimes_{\Q} \Q G\\
&=& F\otimes_{\Q} \R G\\
&=& F\otimes_{\Q} \R \otimes_{\R} \R G\\
&=& (\R^r\oplus \C^s)\otimes_{\R} \R G\\
&=& (\R G)^r \oplus (\C G)^s.
\end{eqnarray*}
 \end{proof}

\begin{proposition}
 Let $G$ be a finite strongly monomial group and $F$ a number field with ring of integers $R$. Let $X$ be a complete set of representatives of strong Shoda pairs of $G$. Then the rank of $\mathcal{Z}(\U(R G))$ equals $(r-1)r_{\R}(G)+sr_{\C}(G)+ \sum_{(H,K)\in X}\frac{\phi([H:K])}{|N_G(K)|}\left(\frac{|H|}{k_{(H,K)}}-\frac{|E_F(G,H/K)|}{|I_{[H:K]}(F)|}\right)$, where $r$ is the number of real embeddings of $F$ and $2s$ is the number of complex embeddings of $F$.
\end{proposition}
\begin{proof}
 From Lemma~\ref{tensor} it follows that the rank of $\mathcal{Z}(\U(R G))$ equals $r r_{\R}(G)+sr_{\C}(G)-r_F(G)$. 
From Proposition~\ref{number}, $$r_F(G)=r_{\Q}(G)+\sum_{(H,K)\in X}\left(\frac{\phi([H:K])}{|I_{[H:K]}(F)|}\frac{|E_F(G,H/K)|}{|N_G(K)|}-1\right).$$ 
Therefore, applying Theorem~\ref{rank},
\begin{eqnarray*}
& & r r_{\R}(G)+sr_{\C}(G)-r_F(G) \\
&&= (r-1)r_{\R}(G)+sr_{\C}(G)+ r_{\R}(G) - r_{\Q}(G) -\sum_{(H,K)\in X}\left(\frac{\phi([H:K])}{|I_{[H:K]}(F)|}\frac{|E_F(G,H/K)|}{|N_G(K)|}-1\right)\\
&&= (r-1)r_{\R}(G)+sr_{\C}(G)+ \sum_{(H,K)\in X} \left(\left(\frac{\phi([H:K])}{k_{(H,K)}[N:H]}-1\right)-\left(\frac{\phi([H:K])}{|I_{[H:K]}(F)|}\frac{|E_F(G,H/K)|}{|N_G(K)|}-1\right)\right)\\
&&= (r-1)r_{\R}(G)+sr_{\C}(G)+ \sum_{(H,K)\in X}\frac{\phi([H:K])}{|N_G(K)|}\left(\frac{|H|}{k_{(H,K)}}-\frac{|E_F(G,H/K)|}{|I_{[H:K]}(F)|}\right).
\end{eqnarray*}
 \end{proof}

\renewcommand{\bibname}{References}
\bibliographystyle{amsalpha}
\bibliography{references}

\providecommand{\bysame}{\leavevmode\hbox to3em{\hrulefill}\thinspace}
\providecommand{\MR}{\relax\ifhmode\unskip\space\fi MR }
\providecommand{\MRhref}[2]{%
  \href{http://www.ams.org/mathscinet-getitem?mr=#1}{#2}
}
\providecommand{\href}[2]{#2}
\begin{thebibliography}{JdROVG14}

\bibitem[BdR07]{Broche2007}
O.~Broche and {\'A}.~del R\'io, \emph{Wedderburn decomposition of finite group
  algebras}, Finite Fields Appl. \textbf{13} (2007), no.~1, 71--79.

\bibitem[CR81]{1981CurtisReiner}
C.W. Curtis and I.~Reiner, \emph{Methods of representation theory with
  applications to finite groups and orders. {V}ol. {I}}, John Wiley \& Sons
  Inc., New York, 1981.

\bibitem[Fer04]{Ferraz2004}
R.A. Ferraz, \emph{Simple components and central units in group algebras}, J.
  Algebra \textbf{279} (2004), no.~1, 191--203.

\bibitem[FP07]{Ferraz2007}
R.A. Ferraz and C.~{Polcino Milies}, \emph{Idempotents in group algebras and
  minimal abelian codes}, Finite Fields Appl. \textbf{13} (2007), no.~2,
  382--393.

\bibitem[GJ98]{GJ}
A.~Giambruno and E.~Jespers, \emph{Central idempotents and units in rational
  group algebras of alternating groups}, Internat. J. Algebra Comput.
  \textbf{8} (1998), no.~4, 467--477.

\bibitem[Isa76]{Isaacs1976}
I.M. Isaacs, \emph{Character theory of finite groups}, Academic Press [Harcourt
  Brace Jovanovich, Publishers], New York-London, 1976, Pure and Applied
  Mathematics, No. 69.

\bibitem[JdROVG13]{2013JdROVG}
E.~Jespers, {\'A}.~del R{\'i}o, G.~Olteanu, and I.~Van~Gelder, \emph{Group
  rings of finite strongly monomial groups: Central units and primitive
  idempotents}, J. Algebra \textbf{387} (2013), 99--116.

\bibitem[JdROVG14]{2012JdROVG}
\bysame, \emph{Central units of integral group rings}, Proc. Amer. Math. Soc.
  \textbf{142} (2014), no.~7, 2193--2209.

\bibitem[JLP03]{Jespers2003}
E.~Jespers, G.~Leal, and A.~Paques, \emph{Central idempotents in the rational
  group algebra of a finite nilpotent group}, J. Algebra Appl. \textbf{2}
  (2003), no.~1, 57--62.

\bibitem[JP12]{JesPar2011}
E.~Jespers and M.M. Parmenter, \emph{Construction of central units in integral
  group rings of finite groups}, Proc. Amer. Math. Soc. \textbf{140} (2012),
  no.~1, 99--107.

\bibitem[JPS96]{JesParSeh1996}
E.~Jespers, M.M. Parmenter, and S.K. Sehgal, \emph{Central units of integral
  group rings of nilpotent groups}, Proc. Amer. Math. Soc. \textbf{124} (1996),
  no.~4, 1007--1012.

\bibitem[OdRS04]{Olivieri2004}
A.~Olivieri, {\'A}.~del R{\'i}o, and J.J. Sim{\'o}n, \emph{On monomial
  characters and central idempotents of rational group algebras}, Comm. Algebra
  \textbf{32} (2004), no.~4, 1531--1550.

\bibitem[OdRS06]{Olivieri2006}
\bysame, \emph{The group of automorphisms of the rational group algebra of a
  finite metacyclic group}, Comm. Algebra \textbf{34} (2006), no.~10,
  3543--3567.

\bibitem[Olt07]{2007Olteanu}
G.~Olteanu, \emph{Computing the {W}edderburn decomposition of group algebras by
  the {B}rauer-{W}itt theorem}, Math. Comp. \textbf{76} (2007), no.~258,
  1073--1087.

\bibitem[Rei75]{Reiner1975}
I.~Reiner, \emph{Maximal orders}, Academic Press, London, New York, San
  Fransisco, 1975.

\bibitem[Sho33]{Shoda1933}
K.~Shoda, \emph{{\"U}ber die monomialen darstellungen einer endlichen gruppe},
  Proc. Phys.-math. Soc. Jap. \textbf{15} (1933), no.~3, 249--257.

\bibitem[Yam73]{Yamada1973}
T.~Yamada, \emph{The {S}chur subgroup of the {B}rauer group}, Lect. Notes Math,
  vol. 397, Springer-Verlag, 1973.

\end{thebibliography}

\end{document}